\documentclass[12 pt]{amsart}

\usepackage{hyperref}
\usepackage{etex}
\usepackage[shortlabels]{enumitem}
\usepackage{amsmath}
\usepackage{amsxtra}
\usepackage{amscd}
\usepackage{amsthm}
\usepackage{amsfonts}
\usepackage{amssymb}
\usepackage{eucal}
\usepackage[all]{xy}
\usepackage{graphicx}
\usepackage{tikz-cd}
\usepackage{mathrsfs}
\usepackage{subfiles}
\usepackage{mathpazo}
\usepackage{euler}
\usepackage[colorinlistoftodos, textsize=tiny]{todonotes}
\usepackage{morefloats}
\usepackage{mathdots}
\usepackage{pdfpages}
\usepackage{thm-restate}
\usepackage[utf8]{inputenc}
\usepackage{epigraph}
\usepackage{csquotes}
\usepackage[margin=1in]{geometry}

\graphicspath{ {images/} }

\RequirePackage{color}
\definecolor{myred}{rgb}{0.75,0,0}
\definecolor{mygreen}{rgb}{0,0.5,0}
\definecolor{myblue}{rgb}{0,0,0.65}

\usepackage{hyperref}
\hypersetup{citecolor=blue}
\usepackage{tikz}
\usetikzlibrary{matrix,arrows,decorations.pathmorphing}

\theoremstyle{plain}
\newtheorem{theorem}[subsection]{Theorem}

\newtheorem{proposition}[subsection]{Proposition}
\newtheorem{lemma}[subsection]{Lemma}
\newtheorem{corollary}[subsection]{Corollary}

\theoremstyle{definition}

\newtheorem{remark}[subsection]{Remark}
\newtheorem{example}[subsection]{Example}

\theoremstyle{remark}
\newtheorem{notation}[subsection]{Notation}

\numberwithin{equation}{section}

\newcommand\nc{\newcommand}
\nc\on{\operatorname}
\nc\renc{\renewcommand}

\newcommand\bp{{\mathbb P}}

\newcommand\bz{{\mathbb Z}}
\newcommand\ba{{\mathbb A}}

\newcommand \ra{\rightarrow}
\newcommand \spec{\text{Spec }}

\newcommand \pt{\pi_1^{\on{tame}}}

\DeclareMathOperator\gal{Gal}

\def\listtodoname{List of Todos}
\def\listoftodos{\@starttoc{tdo}\listtodoname}

\title[Invariance of Fundamental Group]{Invariance of the tame fundamental group under base change between algebraically closed fields}
\author{Aaron Landesman}

\begin{document}

\maketitle

\begin{abstract}
	We show that the tame \'etale fundamental group
	of a connected normal finite type separated scheme
	remains invariant
	upon base change between algebraically closed fields of characteristic
	$p \geq 0$.
\end{abstract}

\section{Statement of theorem}

In a wide range of number theoretic situations, one may want to compare local systems on a
variety over one algebraically closed field to local systems on the base change
of the variety to a larger
algebraically closed field. 
At least when these local systems are tame, the two
notions should be equivalent.
Our main result,
\autoref{theorem:isomorphism-on-algebraically-closed-base-change}, states that this is indeed true. See
\autoref{remark:examples-in-literature} for some sample uses of this result in
number theory.

We now introduce notation to precisely state our main result.
Let $U$ be a connected normal finite type separated scheme over an algebraically closed base field $k$ of characteristic
$p$, allowing the possibility
$p = 0$. Let $\pi_1(U)$ denote the \'etale fundamental group of $U$,
where we leave the base point implicit.
If $\overline{U}$ is a proper normal scheme containing $U$ as a dense open
subscheme, we call $\overline{U}$ a normal compactification of $U$.
If moreover $\overline U$ is projective, we call $\overline U$ a projective
normal compactification of $U$.
Normal compactifications of normal separated finite type schemes always exist, and
projective normal compactifications of normal quasi-projective schemes always
exist, as described in \autoref{remark:tameness-definition}.

We next introduce notation to define the numerically tame fundamental group with respect to the above
normal compactification $U \to \overline{U}$. We denote this by $\pt(U)$, which
implicitly depends on the normal compactification $U \subset \overline{U}$.
See
\cite[Appendix, Example 2]{kerzS:on-different-notions}
for an example demonstrating this dependence on the choice of compactification.
Also see \autoref{remark:numerically-tame-compatibility}.
This numerically tame fundamental group is a quotient of the usual \'etale
fundamental group. Moreover, the prime-to-$p$ \'etale fundamental group, whose
finite quotients correspond to covers of degree relatively prime to $p$, is a quotient of
the tame fundamental group.
Here and elsewhere, when $p = 0$, we consider every integer to be relatively
prime to $p$ so that the prime-to-$p$ \'etale fundamental group is the same as
the usual \'etale fundamental group.

First, we introduce the notion of tameness. In order to define tameness, we
first recall the definition of the inertia group.
Let $E \to U$ be a finite \'etale Galois $G$-cover.
By convention, we assume Galois covers are connected.
Let $s \in \overline{U}$ be a point, let
$\overline{E}$ denote the normalization of $\overline{U}$ in
the function field of $E$.
Let $t \in \overline{E}$ map to $s$ and define the {\em decomposition
group} of $\overline E \to \overline U$ at $t$ to be $D_{t,\overline E/\overline U}:= \{g \in G : gt = t\}$.
Then, the {\em inertia group} of $\overline E \to \overline U$ at $t$ is
$I_{t,\overline E/\overline U} := \ker\left( D_{t,\overline E/\overline U} \to
\on{Aut}_s(t) \right).$
Changing our choice of $t$ results in a conjugate inertia group, and so we use
$I_{s,\overline E/\overline U}$ 
to denote the {\em inertia group} of $\overline E \to \overline U$ at $s$ 
which is the conjugacy class of the subgroup $I_{t,\overline E/\overline U}$ for any $t$ over $s$.
Note that in the case that the residue fields of $s$ and $t$ agree, the
inertia group agrees with the decomposition group. In particular, this
automatically holds when the residue field of $s$ is algebraically closed.

We next define tameness.
We say $E \to U$ is {\em tame along $s$} if the inertia group of $\overline E
\to \overline U$ at
$s$ has order prime to $p$. 
In the case $E \to U$ is not Galois, we say $E \to U$ is {\em tame along $s$} if
the Galois closure of $E \to U$ is tame along $s$.
We say $E \to U$ is {\em tame} if it is tame at every point $s \in
\overline{U} - U$.

Finally, we come to the definition of the numerically tame fundamental group.
Let $\overline{b} \in U$ denote a geometric point, which we use as a basepoint.
For $E \to U$ a finite \'etale Galois cover, let $\on{Hom}_U(\overline{b}, E)$
denote the set of maps $\overline{b} \to E$ whose
composition with $E \to U$ is the given map $\overline{b} \to U$.
Following \cite[\S7, p. 17]{kerzS:on-different-notions}
the {\em numerically tame fundamental group}, $\pt(U, \overline b),$ is by definition 
the automorphism group of the fiber functor which sends
a tame finite \'etale cover $E \to U$ to $\on{Hom}_U(\overline{b}, E)$.
Since every connected finite \'etale cover is dominated by a 
Galois finite \'etale cover, this profinite group is non-canonically in bijection with the
profinite set
$\lim_{\substack {E \to U \\ \text{finite \'etale tame Galois covers}}} \on{Hom}_U(\overline{b}, E),$
and the latter is a torsor under the former, whose trivialization can be
obtained by choosing a compatible system of basepoints in each
$\on{Hom}_U(\overline{b}, E)$.
We remind the reader that $\pt(U,\overline{b})$ implicitly depends
on the choice of normal compactification $U \to \overline{U}$ because the set of
finite \'etale tame Galois covers
implicitly depends on the compactification.
In what follows, we will omit the basepoint $\overline{b}$ from the notation,
and simply write it as $\pt(U)$.
See \cite{schmidt:tame-coverings} and \cite{kerzS:on-different-notions} for more background on the numerically tame fundamental group.
In particular, when $k$ has characteristic $0$, $\pi_1(U) \simeq \pt(U)$.

If $X \ra Y$ and $Z \ra Y$ are morphisms, we denote $X \times_Y Z$ by $X_Z$.
In the case $Z = \spec B$, we also denote $X \times_Y Z$ by $X_B$.

Our main result is the following theorem:
\begin{theorem}
	\label{theorem:isomorphism-on-algebraically-closed-base-change}
	Suppose $k$ is an algebraically closed field of characteristic $p\geq 0$ 
	and $U$ is a connected normal separated finite type scheme over $k$.
Let $L$ be any algebraically closed field containing $k$
and $\overline U$ any normal compactification of $U$.
Then, the natural
map $\pt(U_L) \ra \pt(U)$ is an isomorphism, where tameness for covers of $U_L$
is taken with respect to the normal compactification $U_L \subset (\overline U)_L$.
\end{theorem}
Using the fact that the fundamental group of a scheme is unchanged under
inseparable field extensions, 
\cite[Tag 0BQN]{stacks-project},
we can generalize the above theorem to the case
that $k$ and $L$ are only separably closed.
\begin{corollary}
	\label{corollary:separably-closed}
	Suppose $k$ is a separably closed field of characteristic $p\geq 0$ 
	and $U$ is a connected normal separated finite type scheme over $k$.
Let $L$ be any separably closed field containing $k$
and $\overline U$ any normal compactification of $U$.
Then, the natural
map $\pt(U_L) \ra \pt(U)$ is an isomorphism, where tameness for covers of $U_L$
is taken with respect to the normal compactification $U_L \subset (\overline U)_L$.
\end{corollary}

\begin{remark}
	\label{remark:tame-and-prime-to-p}
	The result
	\autoref{theorem:isomorphism-on-algebraically-closed-base-change} for tame fundamental groups described above
	implies an analogous result for prime-to-$p$ fundamental groups.
	Namely, let $\pi_1'(U)$ denote the prime-to-$p$ fundamental group, which is the
limit of automorphism groups of all Galois finite \'etale covers of
$U$ of degree prime to $p$.
Because prime-to-$p$ covers are all tame, we obtain from
\autoref{theorem:isomorphism-on-algebraically-closed-base-change} that the natural
map $\pi_1'(U_L) \ra \pi_1'(U)$ is an isomorphism.
\end{remark}

\begin{remark}
	\label{remark:attribution}
	\autoref{theorem:isomorphism-on-algebraically-closed-base-change} is surely a folklore theorem.
	Nevertheless, in its complete form, the author was unable to find it in the
	literature.
	The proof written here is primarily a combination of ideas presented to me by 
	Brian Conrad and Jason Starr.
	In particular, Jason Starr has written up a separate proof on mathoverflow at
	\cite{MO:etale-fundamental-group-base-change-between-algebraically-closed-fields}.
	The proof in this note is a re-organization of the ideas presented in that post.
\end{remark}

\begin{remark}
	\label{remark:past-work}
	Many special cases of
	\autoref{theorem:isomorphism-on-algebraically-closed-base-change}
	already exist in the literature.
		The prime-to-$p$ version of
	\autoref{theorem:isomorphism-on-algebraically-closed-base-change} as in
	\autoref{remark:tame-and-prime-to-p}
	was previously verified in
	\cite[Corollary
	A.12]{lieblichO:generators-and-relations-for-the-etale-fundamental-group}
	via a proof heavily involving stacks.
	Separately, this was also shown in \cite[Corollaire
	4.5]{orgogozo:alterations-et-groupe-fondamental}.
	The important special case that $U$ is a curve is also mentioned
	in \cite[Theorem
	6.1]{orgogozo2000theoreme}, though the proof is omitted there.
	In characteristic $0$, a proof is given in
\cite[Expos\'e XIII, Proposition 4.6]{noopsortSGA1Grothendieck1971}
taking $Y = \spec L$ in the statement there.
	However, that proof relies on resolution of singularities.

	In the case $U$ is proper, this was proven in
	\cite[Th\'eor\`eme 3]{langS:sur-les-revetements}, 
	\cite[Proposition 5.6.7]{szamuely:galois-groups-and-fundamental-groups},
	\cite[Expos\'e X, Corollaire 1.8]{sga1},
	and also
	\cite[Tag 0A49]{stacks-project}.
\end{remark}
\begin{remark}
	\label{remark:examples-in-literature}
	\autoref{theorem:isomorphism-on-algebraically-closed-base-change} is
	frequently used in the literature.
	We provide a few such instances we have come across, but expect that
	many more examples exist.

	In the case $U$ is quasi-projective and $k$ has positive characteristic,
	the prime-to-$p$ version as in \autoref{remark:tame-and-prime-to-p} 
	is used in \cite[(4.2.1)]{litt2021arithmetic}
	regarding arithmetic representations of fundamental groups.

	In the case $k$ has characteristic $0$, this result is useful in transferring properties of the fundamental group of a variety over $\mathbb Q$ to the corresponding
	base change to $\mathbb C$.
	For example, this was used in the proof of \cite[Lemma 5.2]{zywina2010hilbert} 
	in order to understand images of
	Galois representations of abelian varieties.
	Another sample use is \cite[p. 701, paragraph 3, proof of Proposition
	4.9]{landesman:geometric-average-selmer}, where the result was used by
	the author to
	estimate average sizes of Selmer groups of elliptic curves over function
	fields.

	As is evident, from the above number theoretic examples, 
	\autoref{theorem:isomorphism-on-algebraically-closed-base-change} crops
	up in a variety of situations relevant to number theorists, and so may
	prove a useful fact in the number theorist's toolkit.
\end{remark}

\begin{example}
	\label{example:failure-of-algebraically-closed-base-change-in-char-p}
	The tameness hypothesis in the characteristic $p > 0$ case is crucial.
	If $k \subset L$ are two algebraically closed fields
	of characteristic $p > 0$, then for $U$ a normal quasi-projective scheme
	over $k$, the map $\pi_1(U_L) \ra \pi_1(U)$ is not in general
	an isomorphism.
	Artin--Schreier covers provide counterexamples in the case
	$U = \ba^1_k$.
	In more detail, if $\pi_1(\ba^1_L) \ra \pi_1(\ba^1_k)$ were an isomorphism,
	then the map $H^1(\ba^1_k, \bz/(p)) \ra H^1(\ba^1_L, \bz/(p))$
	would also be an isomorphism.
	The Artin--Schreier exact sequence identifies this with the
	map
	$k\left[ x \right]/\left\{ f^p - f : f \in k[x] \right\} \ra L\left[ x \right]/\left\{ f^p - f : f \in L[x] \right\}$,
	and this map is not surjective because $a x^{p-1}$ for $a \in L - k$ does not lie in the image.
\end{example}

\begin{remark}
	\label{remark:tameness-definition}
	Note that the standard definition of the tame fundamental group is more
	restrictive than our definition in terms of numerical tameness, because the usual
	definition as in
	\cite[Expose XIII, 2.1.3]{noopsortSGA1Grothendieck1971} assumes $U$ has a smooth 
	compactification whose boundary is a normal crossings divisor.
	With this notion from \cite{noopsortSGA1Grothendieck1971}, the tame
	fundamental group is independent of the choice of compactification.

	In contrast, the notion of tame fundamental group we use here makes
	sense for any normal finite type separated scheme $U$ over $k$, since we can find a
	normal compactification of $U$ as follows:

	By Nagata compactification, 
\cite[\href{https://stacks.math.columbia.edu/tag/0F41}{Tag 0F41}]{stacks-project}
if $U$ is finite type and separated, there exists a quasi-compact open immersion $U \to \overline{U}$, where
$\overline{U}$ is a proper scheme. One can then replace $\overline{U}$ with its
normalization to obtain a proper normal scheme $\overline{U}$, containing $U$ as
a dense open.

Moreover, in the case $U$ is quasi-projective, we can also assume $\overline{U}$ is
projective by taking any projective scheme $\overline U$ containing $U$
as a dense open and then replacing $\overline U$ by its normalization.
\end{remark}
\begin{remark}
		\label{remark:numerically-tame-compatibility}
		Our notion of the numerically tame fundamental group agrees with the usual notion
	described in \cite[Expose XIII, 2.1.3]{noopsortSGA1Grothendieck1971}
	when the compactification of $U$
	is smooth with normal crossings boundary by
	\cite[Proposition 1.14]{schmidt:tame-coverings}.
	This tame fundamental group is not in general independent of the choice
	of normal compactification, see \cite[Appendix, Example
	2]{kerzS:on-different-notions}.
\end{remark}

\subsection{Acknowledgements}

I would like to thank Brian Conrad, Jason Starr, and Sean Cotner for key ideas in the proof.
I also thank several anonymous referees for numerous incredibly thorough
readings, many extremely
helpful comments, and multitudes of thoughtful suggestions.
Additionally, I thank Sean Cotner for a detailed reading, and thorough comments.
I thank Peter Haine, Daniel Litt, Martin Olsson, 
Tam\'as Szamuely,
and further anonymous referees for helpful comments.
This material is based upon work supported by the National Science Foundation Graduate Research Fellowship under Grant No. DGE-1656518.

\section{Proof of theorem}
\subsection{Idea of proof of \autoref{theorem:isomorphism-on-algebraically-closed-base-change}}
\label{subsubsection:idea-of-proof-of-algebraically-closed-base-change}
The proof of \autoref{theorem:isomorphism-on-algebraically-closed-base-change} is fairly technically involved,
but the idea is not too complicated:
The key is to verify injectivity of $\pt(U_L) \ra \pt(U)$.
As a first step, we reduce from the normal case to the smooth case using that geometrically normal schemes have a dense open smooth subscheme.
Then, using Chow's lemma,
we reduce to the smooth quasi-projective case.
We therefore assume our variety $U$ is smooth and quasi-projective, and prove the theorem by reducing it to the curve case.
For this reduction, we fiber $U$ over a variety of one lower dimension, in which case we can apply the curve case
to the geometric generic fiber of the fibration.

It remains to deal with the case that $U$ is a quasi-projective smooth curve, which
is also the most technically involved part.
In this case, we can write $U$ as $\overline U - D$,
with $\overline U$ smooth and projective and $D$ a divisor.
To check injectivity, we want to check every finite \'etale cover of $U_L$ is the base change of some finite \'etale cover of $U$.
If $E$ is one such cover, we can use spreading out and specialization to obtain an \'etale cover $U' \ra U$
with the same ramification index over each point of $D$ that $E$ has.
Then, we construct the cover $E'$ which is the normalization of $E$ in $E \times_{U_L} U_L'$,
and verify this is the base change of a cover from $k$.
We do so
by applying the projective version of \autoref{theorem:isomorphism-on-algebraically-closed-base-change},
using that $E'$ and $U_L'$ have projective compactifications $\overline E'$ and $\overline U_L'$
with a finite \'etale map $\overline E' \ra \overline U_L'$.

We now indicate how we put together the steps described in the above to prove \autoref{theorem:isomorphism-on-algebraically-closed-base-change}.
In \autoref{subsection:surjectivity} (\autoref{lemma:surjectivity}), we prove $\pt(U_L) \ra \pt(U)$ is surjective.
For injectivity, we first prove the map is injective
in the case $U$ is a smooth, connected, and quasi-projective curve in
\autoref{subsection:curve-case} (\autoref{proposition:curve-case}).
We prove in \autoref{subsection:induction}
(\autoref{proposition:smooth-quasi-projective-case}) that \autoref{theorem:isomorphism-on-algebraically-closed-base-change}
holds for smooth, quasi-projective varieties of all dimensions.
We next verify the case that $U$ is smooth, finite type, and separated in
\autoref{proposition:smooth-finite-type-case}.
Finally, we complete the proof
in the case that $U$ is normal, connected, finite type, and separated in \autoref{subsubsection:normal-case}.

\subsection{Surjectivity}
\label{subsection:surjectivity}
We first show $\pt(U_L) \ra \pt(U)$ is surjective.
\begin{lemma}
	\label{lemma:surjectivity}
	The map $\pi_1(U_L) \to \pi_1(U)$ is surjective. In particular, $\pt(U_L) \ra \pt(U)$ is surjective.
\end{lemma}
\begin{proof}
	It suffices to verify
	that the pullback of any connected finite \'etale cover over $U$ along
	$U_L \ra U$ is connected,
	see, for example, \cite[\href{https://stacks.math.columbia.edu/tag/0BN6}{Tag 0BN6}]{stacks-project}.
	Since $L$ and $k$ are both algebraically closed, 
	the result follows from the fact that connectedness is preserved under base change between algebraically closed
	fields \cite[Proposition 4.5.1]{EGAIV.2}.
\end{proof}

\subsection{Proof of injectivity in the curve case}	
\label{subsection:curve-case}
	We next prove injectivity for smooth connected quasi-projective curves $U$.	
	For this, it suffices to show that any tame Galois finite \'etale 
	cover $E$ of $U_L$ is the base change of some tame Galois finite \'etale cover of $U$.
	Note that any such cover of $U$, whose base change is a tame cover $E$ of
	$U_L$, is automatically tame, since
	tameness can be verified after base extension.
	To prove such an $E$ exists, it suffices
	to find a connected finite \'etale cover $F' \ra U$ over $k$ so that
	$F'_L \ra U_L$ factors through $E$. 

	As a first step, we wish to find a cover $U'$ of $U$ with the same
	ramification indices as $E$ over points
	in the normal projective compactification of $U$.

	\begin{notation}
\label{subsubsection:notation-E}
	Let $k \ra L$ be an inclusion of algebraically closed fields,
	let $U$ be a smooth curve over $k$, $\overline U$ its regular
	projective compactification, and $D := \overline U - U$.
	Let $E \rightarrow U_L$ be a tame Galois finite \'etale cover.
	Let $\overline E$ be the normalization of $\overline U_L$ inside $E$.	
\end{notation}

	\begin{lemma}
		\label{lemma:ramification-order-cover}
		With notation as in \autoref{subsubsection:notation-E},
		there exists a finite Galois cover $\overline U' \ra \overline U$, \'etale over $U$,
		with the same ramification indices that $\overline E$ has over
		the corresponding points of $D_L$.
	\end{lemma}
	The idea of this proof is to ``spread out and specialize'' $E$. See
	\eqref{equation:specialization-setup} for a diagram.
	\begin{proof}
		To construct $\overline{U}'$, we can find a finitely generated $k$-subalgebra
	$A \subset L$ and a finite \'etale cover $E_A \ra U_A$, over $A$
	so that $(E_{A})_L \simeq E$ and $E_A \to U_A$ is finite \'etale Galois
	and tame.
	Let $\overline E_A$ denote the normalization of $\overline U_A$ along
	$E_A \to U_A$.
	Note that, because of the Galois condition,
	the ramification index of a point of $\overline E_A$ over a point of
	$\overline U_A$ only
	depends on the image point in $\overline U_A$. We may therefore speak of the
	ramification index over a point of $\overline U_A$.
	Since $k$ is algebraically closed, for any field $K \supset k$, the irreducible components of
	$D_K$ arise uniquely from the irreducible components of $D$ under scalar
	extension. We freely use the above observations
	in what follows.

	Let $K(A)$ denote the fraction field of $A$.
	Note that the ramification index of $E_{K(A)}$ over each point of $D_{K(A)}$ 
	agrees with that of $E$ over the corresponding point of $D_L$.
	Further, we claim that for a general closed point $s$ of $\spec A$,
	the ramification index of $s \times_{\spec A} E_A$ over a point of $s \times_{\spec A} D_A \simeq D$
	agrees with the ramification index of $E_{K(A)}$ over the corresponding generic point of $D_{K(A)}$.
	To see why this ramification index $n$ is constant over an open set of $\spec A$, 
	recall that we are assuming the cover $E \to U$ is tame, and so, after
	possibly shrinking $\spec A$, we may assume the same of $E_A \to U_A$.
	By the tameness hypothesis, the ramification index over a point can be identified with one more 
	than the degree of the relative sheaf of differentials at that point,
	see, for example, \cite[p. 592]{vakil:foundations-of-algebraic-geometry-2017}.
	(The point here is that if the map is locally of the
	form $t \mapsto us^n$, for $t$ and $s$ uniformizers and $u$ a unit, then the
	derivative is $dt = d(us^n) = un s^{n-1} ds + s^n du$, which has
	order precisely $n-1$ if $n$ is not divisible by the
	characteristic.)
	So, for $p \in D$ a geometric point, under the identification $p_A \simeq \spec A$, we see that at any point
	of $\overline{E}_A \times_{\overline{U}_A} p_A$ over the generic point
	of $\spec A$, 
	$\Omega_{\overline{E}_A	\times_{\overline{U}_A} p_A / p_A}$ has degree $n-1$.
	It follows that there is a nonempty open subscheme of $\spec A$ where $\Omega_{\overline{E}_A	\times_{\overline{U}_A} p_A / p_A}$
	has degree $n-1$.
	Hence, the morphism has inertia of order $n$ over some open subscheme of $\spec A$.

	Since $k$ is an algebraically closed field, every closed point of $\spec A$ has residue field $k$, so we may
	choose such a closed point $t \colon \spec k \ra \spec A$
	with the same ramification indices over $D$ as $E$ has over the corresponding points of $D_L$.
	Since the locus of geometric points on the base $\spec A$ where the map $E_A \ra U_A$ is a
	map of connected schemes is constructible \cite[Corollaire
	9.7.9]{EGAIV.3}, we may also assume
	the fiber of $E_A \ra U_A$ over $t: \spec k \to \spec A$ is connected.
	Then, $U' := E_A \times_{\spec A} \spec k$ is our desired connected finite \'etale cover.
	Finally, we take $\overline{U}'$ to be the normalization of
	$\overline{U}$ along $U' \to U$.
	\end{proof}
	
	Summarizing the situation of \autoref{lemma:ramification-order-cover},
	we obtain the commutative diagram
\begin{equation}
	\label{equation:specialization-setup}
	\begin{tikzcd} 
		E \ar{r} \ar{d} & E_A \ar{d} & U'\ar{l}\ar{d} \\
		U_L \ar{d} \ar{r} & U_A \ar{d}  & U \ar{l}\ar{d} \\
		\spec L \ar{r} & \spec A & \spec k \ar[swap]{l}{t},
	\end{tikzcd}\end{equation}
where the four squares are fiber products.

\begin{notation}
\label{subsubsection:notation-U'}
Let $\overline U' \to \overline{U}$ denote the finite Galois cover of
\autoref{lemma:ramification-order-cover}.
	Let $\overline E'$ denote the normalization of $\overline E$ in $\overline E \times_{\overline U_L} \overline U_L'$ 
	and let $E' := \overline E' \times_{\overline U_L'} U_L'$,
	as in the commutative diagram
	  \[\begin{tikzcd}[row sep={40,between origins}, column sep={40,between origins}]
      & E' \ar{rr}\ar{dd}\ar{dl} & & \overline E' \ar{dd}\ar{dl} \\
    E \ar[crossing over]{rr} \ar{dd} & & \overline E \\
      & U_L' \ar{rr} \ar{dl} & & \overline U_L' \ar{dl} \\
      U_L \ar{rr} && \overline U_L \ar[from=uu,crossing over] & & D_L. \ar{ll}
\end{tikzcd}\]
\end{notation}

\begin{remark}
	\label{remark:abhyankar}
	Observe that the finite map $\overline U' \ra \overline U$ of
	\autoref{subsubsection:notation-U'} restricts to $U' \ra U$ over $U \subset \overline U$ as $U$ is normal.
By Abhyankar's lemma
\cite[A I.11]{FreitagK:lectures-etale}
(see also \cite[Expose XIII, 5.2]{noopsortSGA1Grothendieck1971})
we obtain that $\overline U'$ is regular, hence smooth, as we are working over an algebraically closed field $k$.
\end{remark}

Although the normalization $\overline{E} \ra \overline{U}_L$ of $\overline{U}_L$ in $E \ra U_L$
is not necessarily \'etale, we now show the finite surjection $\overline E' \ra \overline U_L'$ is \'etale.
\begin{lemma}
	\label{lemma:projective-cover-is-etale}
	With notation as in \autoref{subsubsection:notation-E} and
	\autoref{subsubsection:notation-U'},
	$\overline E' \ra \overline U_L'$ is \'etale.
\end{lemma}
\begin{proof}
Since $E' \to U_L'$ is \'etale by construction, it is enough to check $\overline E' \ra \overline U_L'$
is \'etale over all points of $\overline U_L'$ lying above a point of $D_L$.
Indeed, this is where we crucially use the assumption that $E \ra U$ is tame.
Since being \'etale can be checked
in the local ring at each such point, \'etaleness of $\overline E' \ra\overline U_L'$ follows from 
a version of Abhyankar's lemma, using that the ramification orders of $\overline U_L' \to
\overline U_L$ and $\overline E \to \overline U_L$ agree over each point of
$D_L$, by \autoref{lemma:ramification-order-cover}. For a precise form of Abhyankar's lemma applicable in this setting, see, for example,
\cite[Tag 0EYH]{stacks-project}.
\end{proof}

We are now prepared to complete the curve case of \autoref{theorem:isomorphism-on-algebraically-closed-base-change}.
\begin{proposition}
	\label{proposition:curve-case}
	\autoref{theorem:isomorphism-on-algebraically-closed-base-change} holds
	in the case that $U$ is a smooth connected curve.
\end{proposition}
\begin{proof}
	Let $U$ be a smooth connected curve.
	We use notation from \autoref{subsubsection:notation-E} and \autoref{subsubsection:notation-U'}.
	By \autoref{lemma:surjectivity}, we only need to check injectivity of
	$\pt(U_L) \to \pt(U)$.
	Since $E' \ra U_L$ is a finite \'etale cover of $U_L$ dominating $E \ra U_L$,
	to complete the proof in the case that $U$ is a smooth curve, it suffices to
	show $E' \ra U_L$ is the base change of some tame finite \'etale cover $F'
	\ra U$ over $k$.
	Note here that tameness of $F' \to U$ is
	automatic once we show it base changes to $E' \ra U_L$,
	as tameness can be verified after base extension.
	We showed in \autoref{lemma:projective-cover-is-etale} that $\overline E' \rightarrow \overline U'_L$ is a finite \'etale cover.
	Since $\overline U'$ is projective and normal, by 
	\cite[Th\'eor\`eme 3]{langS:sur-les-revetements},
	we obtain that there is some finite \'etale cover $\overline F' \ra
	\overline U'$ over $k$ with $\overline E' \simeq \left( \overline F' \right)_L$.
	(Alternatively, see \cite[Proposition
	5.6.7]{szamuely:galois-groups-and-fundamental-groups},
	\cite[Expos\'e X, Corollaire 1.8]{sga1},
and \cite[Tag 0A49]{stacks-project}.)
	We then find $F' := \overline F' \times_{\overline U'} U'$ is a finite
	\'etale cover of $U$ satisfying $(F')_L	\simeq E'$, and so this is the
	desired cover.
\end{proof}

\subsection{Dominating compactifications}

In order to complete the reduction from the higher dimensional case to the curve
case, we will want to know that
\autoref{theorem:isomorphism-on-algebraically-closed-base-change} holds for one
compactification $U \to Y$ whenever it holds for another
compactification $U \to X$ with a compatible map $X \to Y$.
The next couple lemmas are devoted to verifying this.
\begin{lemma}
	\label{lemma:tameness-on-covering-compactification}
	Suppose $W$ is a connected smooth separated finite type scheme over a field $k$
	and $\beta: W \subset X$ and $\alpha: W \subset Y$ are two normal
	compactifications with a map $f: X \to Y$ so that $\alpha = f \circ
	\beta$.
	If a finite \'etale Galois cover $E \to W$ is tame with respect to
	$\alpha$, it is also tame with respect to $\beta$.
\end{lemma}
\begin{proof}
	Tameness can be checked after field extension, so we will assume $k$ is
	algebraically closed.
	Fix a point $s \in Y$ and a preimage $t \in X$ with $f(t) = s$.
	Let $F_Y$ denote the normalization of $Y$ in the function field of $E$
	and let $F_X$ denote the normalization of $X$ in the function field of
	$E$. We assume $F_Y$ is tame over $s$ and wish to show $F_X$ is tame
	over $t$.

	We next claim that there is a map $F_X \to F_Y$.
	Let $F$ denote the normalization of $F_Y \times_Y X$. It is enough to
	show the natural map $F \to F_X$ induced by the universal property of
	normalization is an isomorphism, as we then obtain a map $F_X \simeq F
	\to F_Y \times_Y X \to F_Y$. Because the normalization map is finite
	by
\cite[\href{https://stacks.math.columbia.edu/tag/03GR}{Tag 03GR}]{stacks-project}
and
\cite[\href{https://stacks.math.columbia.edu/tag/035B}{Tag
035B}]{stacks-project},
both $F$ and $F_X$ are finite over $X$.
Therefore, the map $F \to F_X$ is a birational map which is finite (because it
is quasi-finite and proper) between normal schemes over
$k$. It follows from a version of Zariski's main theorem that $F \to F_X$ is an
isomorphism \cite[\href{https://stacks.math.columbia.edu/tag/0AB1}{Tag
0AB1}]{stacks-project}.

We now conclude the proof.
Let $v$ be a point of $F_X$ over $t$ and $u \in F_Y$ be the image of $v$ under
the map $F_X \to F_Y$.
Since $v$ maps to $u$, we have an inclusion of decomposition groups $D_{v,F_X/X}
\subset D_{u, F_Y/Y}$.
Since we are assuming $k$ is algebraically closed, this is identified with an
inclusion of inertia groups $I_{v,F_X/X} \subset I_{u, F_Y/Y}$.
Hence, up to conjugacy, the inertia group at $t$ is a subgroup of the inertia
group at $s$ and so tameness at $s$ implies tameness at $t$.
\end{proof}

\begin{lemma}
	\label{lemma:dominating-compactification}
	With the same notation as in
	\autoref{lemma:tameness-on-covering-compactification}, if
	\autoref{theorem:isomorphism-on-algebraically-closed-base-change} holds
	with respect to the normal compactification $W \subset X$,
	\autoref{theorem:isomorphism-on-algebraically-closed-base-change} also
	holds with respect to the compactification $W \subset Y$.
\end{lemma}
\begin{proof}
	By \autoref{lemma:surjectivity}, it suffices to verify injectivity for
	the map $\pt(W_L) \ra \pt(W)$ with respect to the compactification $W
	\subset Y$.
	Using 
	\cite[Corollary 5.5.8]{szamuely:galois-groups-and-fundamental-groups},
	we can rephrase this as showing that if $k \subset L$ is an
	extension of algebraically closed fields and $E \to W_L$ is any
	tame (with respect to $W \to Y$)
	finite \'etale Galois cover, then $E$ arises as the base
	change of a cover $F \to W$ over $k$.
	By
	\autoref{lemma:tameness-on-covering-compactification},
	this cover is also tame with respect to the normal compactification $W \to X$.
	By assumption
	\autoref{theorem:isomorphism-on-algebraically-closed-base-change} holds
	for the compactification $W \to X$, and so 
	$E \to W_L$ is the base
	change of a cover $F\to W$ over $k$, as we wished to show.
\end{proof}

\subsection{Proof of injectivity in the smooth and quasi-projective case}
\label{subsection:induction}
In this section, specifically in
\autoref{proposition:smooth-quasi-projective-case},
we prove \autoref{theorem:isomorphism-on-algebraically-closed-base-change} in
the case that $U$ is a smooth connected quasi-projective variety.
To start, we use Bertini's theorem to obtain a fibration away from a codimension
$2$ subset of $U$. This fibration will allow us to run an induction on the
dimension.

\begin{proposition}
	\label{proposition:smooth-morphism-reduction}
	Let $U$ be a smooth connected quasi-projective variety of
	dimension $d> 1$.
	Choose a projective normal compactification $U \subset \overline{U}$.
	There is a closed subscheme $Z \subset U$ of codimension at least $2$
	and a 
	projective normal compactification $U - Z \to X$ satisfying the
	following three properties:
	\begin{enumerate}
		\item The closed subscheme $Z$ lies in the smooth locus of
			$\overline{U}$.
		\item 	There is a map $X \to
	\overline{U}$ so that the composition $U - Z \to X \to \overline{U}$ agrees with the
	composition $U - Z \to U \to \overline{U}$.
\item	There is a 
	dominant generically smooth map $\alpha \colon X \to
	\bp^{d-1}_k$ with geometrically irreducible generic fiber.
	\end{enumerate}
\end{proposition}
\begin{proof}
	Let
$U \subset \overline U$ be the given projective normal compactification.
Choose an embedding $U \subset \overline U \subset
	\mathbb P^n_k$.
Replacing $\mathbb P^n_k$ by the span of $\overline U$ in $\mathbb P^n_k$, we
may also assume $\overline U$ is nondegenerate.
Choose a 
general codimension $d$ plane $H \subset \bp^n_k$
such that $H \cap \overline U$ is smooth of dimension $0$,
$H \cap (\overline U  -U) = \emptyset$,
and so that, if $J' \subset \bp^n_k$ is a general codimension $d-1$ plane containing $H$,
we have $J' \cap \overline U$ is smooth and geometrically irreducible of
dimension $1$. This is possible because $\overline U$ is normal, hence smooth
away from codimension $2$, and by Bertini's
theorem, as in
\cite[Theoreme 6.10(2) and (3)]{jouanolou1982theoremes}.

Define $Z:= H \cap U = H \cap \overline{U}$ for $H$ as in the previous paragraph.
For $H$ general as above, the following three conditions are satisfied:
$Z \subset U$ has codimension at
least $2$,
$Z$ does not meet $\overline{U} - U$,
and, for a general plane $J'$ containing $H$,
the intersection $J' \cap \overline U$ is smooth and geometrically irreducible.
The second property verifies condition $(1)$ in the statement because it shows $Z \subset U \subset
\overline{U}$ and $U$ is contained in the smooth locus of $\overline{U}$.
Take $X \to \overline{U}$ to be the blow up of $\overline{U}$ along $Z \subset \overline U$.
This verifies condition $(2)$ in the statement.

To conclude, we will show condition $(3)$ in the statement holds. Namely, we
will show there is a dominant map $X \to \bp^{d-1}_k$
whose generic fiber is smooth and geometrically irreducible.
Geometrically, this map is induced by projection of $\overline U$ away from the plane $H$, and
sends a point $x \in U-Z$ to $\on{Span}(x,
H)$, where we view $\on{Span}(x, H)$ as a point of $\mathbb P^{d-1}_k$
parameterizing codimension
$d-1$ planes $J' \subset \bp^{n}_k$ containing $H$.
The above-described map $U-Z \to \bp^{d-1}_k$ extends to a map on the blow up $X =
\on{Bl}_{\overline U \cap H} \overline U \to \bp^{d-1}_k$,
where the fiber over a point $[J'] \in \bp^{d-1}_k$ (parameterizing codimension
$d-1$ planes $J' \subset \bp^{n}_k$ containing $H$) is $J' \cap \overline U$.
By construction of $H$ so that $J' \cap \overline U$ is smooth and geometrically
irreducible for a general codimension $d - 1$ plane $J' \subset \bp^{d-1}_k$
containing $H$, the generic fiber
of the map $X \to \mathbb P^{d-1}_k$ is smooth and geometrically
irreducible.
\end{proof}

Assuming we have a fibration as in \autoref{proposition:smooth-morphism-reduction}, we next show that the fiber
of a tame Galois finite \'etale cover $E \ra U_L$,
when restricted to the generic point of $\mathbb P^{d-1}_L$, is the base change of a Galois 
finite \'etale cover over the generic point of $\mathbb P^{d-1}_k$.

\begin{proposition}
	\label{proposition:base-change-generic-fiber}
	Assume $U$ is a smooth connected $k$-variety of dimension $d \geq 1$
	with a normal projective compactification $U \to \overline U$ and a dominant generically smooth map
	$\alpha \colon \overline U \ra \bp^{d-1}_k$ with geometrically irreducible generic
	fiber.
Let $\eta_k$ denote the generic point of $\bp^{d-1}_k$ and $\eta_L$ denote the geometric
generic point of $\bp^{d-1}_L$.
Any given tame finite \'etale Galois cover $E \to U_L$ restricts to a Galois
finite \'etale cover $E_{\eta_L} \ra U_{\eta_L}$ (with respect to the
compactification $U_{\eta_L}
\subset \overline U_{\eta_L}$)
which is the base change of some Galois finite \'etale cover $F_{\eta_k} \ra U_{\eta_k}$.
\end{proposition}
\begin{proof}
Let $\overline \eta_k$ and $\overline \eta_L$ denote 
compatible algebraic geometric generic points of $\mathbb P^{d-1}_k$ and $\mathbb
P^{d-1}_L$, with
corresponding generic points $\eta_k$ and
$\eta_L$.
By this, we mean that $\overline \eta_k$ has residue field which is the algebraic closure of
$\kappa(\eta_k)$ and similarly for $L$. 
Moreover, they are compatible in the sense that we specify an embedding
$\kappa(\overline \eta_k) \to \kappa(\overline \eta_L)$ restricting to the
inclusion $\kappa (\eta_k) \to \kappa (\eta_L)$.
Let $E_{\overline \eta_L} := E \times_{\bp^{d-1}_L} \overline \eta_L$, which we
note is smooth and of dimension $1$.
Because $E \to U_L$ is tame with respect to $U_L \to \overline U_L$, we obtain
that $E_{\overline \eta_L} \to U_{\overline \eta_L}$ is tame with respect to
$U_{\overline \eta_L} \to \overline U_{\overline \eta_L}$.
By the curve case of \autoref{theorem:isomorphism-on-algebraically-closed-base-change}, shown in
\autoref{proposition:curve-case},
$E_{\overline \eta_L}$ arises as the base change of some cover
$F_{\overline \eta_k} \ra U_{\overline \eta_k}$.
That is, $(F_{\overline \eta_k})_{\overline \eta_L} \simeq E_{\overline \eta_L}$.

To conclude the proof, we only need realize $F_{\overline \eta_k} \ra U_{\overline \eta_k}$
as the base change of a map over $\eta_k$ so that the above isomorphism
$(F_{\overline \eta_k})_{\overline \eta_L} \simeq E_{\overline \eta_L}$
is the base change of an isomorphism over $\eta_L$.
For $K$ a field, we use $K^s$ to denote its separable closure.
We can realize $\overline \eta_k \ra \eta_k$ as the composition of a purely inseparable
morphism $\overline \eta_k \ra \eta_k^s$ and a separable morphism $\eta_k^s \ra \eta_k$ by taking
$\eta_k^s := \spec \kappa(\eta_k)^s$. 
Since $\overline \eta_k \ra \eta_k^s$ is a universal homeomorphism, the same is true of
$U_{\overline \eta_k} \ra U_{\eta_k^s}$, and so the map induces an isomorphism of \'etale fundamental
groups
$\pi_1(U_{\overline \eta_k}) \ra \pi_1(U_{\eta_k^s})$ \cite[Tag 0BQN]{stacks-project}.
It follows that $F_{\overline \eta_k} \ra U_{\overline \eta_k}$ is the base change of a morphism
$F_{\eta_k^s} \ra U_{\eta_k^s}$ over $\eta_k^s$.
Moreover, by spreading out, there is a finite Galois extension
$\eta_k' \to \eta_k$ so that
$F_{\eta_k^s} \ra U_{\eta_k^s}$
is the base change of a morphism
$F_{\eta_{k}'} \ra U_{\eta_{k}'}$ over $\eta_{k}'$.
We next want to verify this is the base change of a map over $\eta_k$,
which we will do by producing descent data along the extension $\eta_k' \to
\eta_k$.

We next set up notation for descent data.
Observe that $\eta_k \simeq \spec k(x_1,
\ldots, x_n)$ and $\eta_L \simeq \spec L(x_1, \ldots, x_n)$.
Let $M := \Gamma(\eta_k', \mathscr O_{\eta_k'})$ so that $\eta_k' = \spec M$.
It follows that the two maps of schemes $\eta_k' \to \eta_k$ and $\eta_L \ra \eta_k$
correspond to the extensions of fields
$k(x_1, \ldots, x_n) \to M$ and $k(x_1, \ldots, x_n) \ra L(x_1, \ldots, x_n)$.
It is a standard fact that these are linearly disjoint, see \autoref{lemma:linearly-disjoint}.
Let $M_L := M \otimes_k L$.
Since $M$ and $L(x_1, \ldots, x_n)$ are linearly disjoint,
base extension defines a bijective map 
\begin{align*}
	\gal(M/k(x_1, \ldots, x_n)) \simeq
\gal(M_L/L(x_1, \ldots, x_n)).
\end{align*}
We denote the above Galois group by $G$.
As described in
\cite[\S6.2, Example B]{BoschLR:Neron},
specifying descent data for
$F_{\eta_k'}\to U_{\eta_k'}$ along $\eta_k' \to \eta_k$,
is equivalent to specifying
an isomorphism $\phi_{F,k,\sigma}: F_{\eta_k'} \to
F_{\eta_k'}$
for each $\sigma \in G$,
defining an action of $G$ on $F_{\eta_k'}$. (We warn
	the reader that the action
is only defined over $\eta_k$ and not over $\eta_k'$.)
Since $U_{\eta_k'}$ is the base change of $U_{\eta_k}$, we do have descent data
$\phi_{U,k,\sigma} : U_{\eta_k'} \to U_{\eta_k'}$.
The descent data $\phi_{F,k,\sigma}$ we wish to produce should live over the descent data for
$\phi_{U,k,\sigma}$, in the sense the diagram
\begin{equation}
	\label{equation:}
	\begin{tikzcd} 
		F_{\eta_k'} \ar {r}{\phi_{F,k,\sigma}} \ar {d} & F_{\eta_k'} \ar {d} \\
			U_{\eta_k'} \ar {r}{\phi_{U,k,\sigma}} & U_{\eta_k'}
\end{tikzcd}\end{equation}
should commute.
Let $\eta_L' := \eta_k' \times_{\eta_k} \eta_L$.
Since we do have descent data for $F_{\eta_L'} \to U_{L'}$ along $\eta_L' \to
\eta_L$,
we have $\phi_{F,L,\sigma}$ and $\phi_{U,L,\sigma}$ so that
\begin{equation}
	\label{equation:}
	\begin{tikzcd} 
	F_{\eta_L'} \ar {r}{\phi_{F,L,\sigma}} \ar {d} & F_{\eta_L'} \ar {d} \\
	U_{\eta_L'} \ar {r}{\phi_{U,L,\sigma}} & U_{\eta_L'}
\end{tikzcd}\end{equation}
commutes.

We wish to show that $\phi_{F,L,\sigma}$ is the base change of a unique map
$\phi_{F,k,\sigma}$ along $\spec L \to \spec k$.
Indeed, consider the $\eta_k$ scheme $\on{Aut}_{\phi_{U,k,\sigma}}(F_{\eta_k'})$ of
automorphisms of $F_{\eta_k'}$ over the specified automorphism 
$\phi_{U,k,\sigma}$ of $U_{\eta_k'}$.
Note that
$\on{Aut}_{\phi_{U,k,\sigma}}(F_{\eta_k'}) \times_{\eta_k} \eta_L \simeq
\on{Aut}_{\phi_{U,L,\sigma}}(F_{\eta_L'})$.
Moreover, for $N \in \{k,L\}$, since the automorphisms of $F_{\eta_N'}$ over $\phi_{U,N,\sigma}$ are given by
composing any given automorphism over $\phi_{U,N,\sigma}$ with an automorphisms of $F_{\eta_N'}$
over $U_{\eta_N'}$, 
$\on{Aut}_{\phi_{U,k,\sigma}}(F_{\eta_k'})$ and
$\on{Aut}_{\phi_{U,L,\sigma}}(F_{\eta_L'})$ are both $G$ torsors.
Since the residue field of
each point of $\on{Aut}_{\phi_{U,k,\sigma}}(F_{\eta_k'})$
over $\eta_k$
is linearly disjoint from the field extension
$\kappa(\eta_k) \to \kappa(\eta_L)$
by \autoref{lemma:linearly-disjoint},
there is a bijection between the points of
$\on{Aut}_{\phi_{U,k,\sigma}}(F_{\eta_k'})$
and $\on{Aut}_{\phi_{U,L,\sigma}}(F_{\eta_L'})$.
Since the latter is the trivial $G$ torsor,
we also obtain
$\on{Aut}_{\phi_{U,k,\sigma}}(F_{\eta_k'})$ is the trivial $G$ torsor.
In other words there is a unique map $\phi_{F,k,\sigma}$ 
over $\phi_{U,k,\sigma}$
whose base change to
$\eta_L$
is $\phi_{F,L,\sigma}$.
Choosing these $\phi_{F,k,\sigma}$ whose base change is
$\phi_{F,L,\sigma}$,
we find that the $\phi_{F,k,\sigma}$ define descent data
(because the
$\phi_{F,L,\sigma}$ do).
Hence, $F_{\eta_k'} \to U_{\eta_k'}$ is the base change of a map
$F_{\eta_k} \to U_{\eta_k}$, as desired.
\end{proof}

We now complete the proof of
\autoref{theorem:isomorphism-on-algebraically-closed-base-change} in the case
$U$ is smooth quasi-projective with a projective normal compactification.
Since $U$ is quasi-projective, recall that such a projective normal compactification exists by
\autoref{remark:tameness-definition}.
\begin{proposition}
	\label{proposition:smooth-quasi-projective-case}
	\autoref{theorem:isomorphism-on-algebraically-closed-base-change} holds
	when $U \to \overline U$ is a projective normal compactification and $U$ is smooth and quasi-projective.
\end{proposition}
\begin{proof}
	The case $d = 1$ holds by \autoref{proposition:curve-case}, and $d = 0$
	is trivial, so we now
	assume $d > 1$.

	By \autoref{proposition:smooth-morphism-reduction}, there is a $Z
	\subset U
	\subset \overline{U}$ and a projective normal compactification $U - Z \to X$
	satisfying the properties given
	there. Then, since $Z$ as in
\autoref{proposition:smooth-morphism-reduction}
has codimension at least $2$,
$\pt(U - Z) \simeq \pt(U)$ because
the tame fundamental group of a smooth variety is unchanged by removing any set of codimension
at least $2$, as shown in \autoref{lemma:removing-codimension-2}. 
Above, the tameness conditions
for both schemes $U - Z$ and $U$ are taken with respect to the projective normal compactification $\overline U$.

Observe that $Z$ is in the smooth locus of $\overline U$ 
by \autoref{proposition:smooth-morphism-reduction}(1)
and $U \to X$ is
a normal compactification of $U$. 
Using \autoref{proposition:smooth-morphism-reduction}(2)
to verify the hypotheses of
\autoref{lemma:dominating-compactification},
it suffices to prove
\autoref{theorem:isomorphism-on-algebraically-closed-base-change} 
for the compactification
$U-Z \to X$ in place of $U \to \overline{U}$.

For the remainder of the proof, we now rename $U-Z$ as $U$ and $X$ as
$\overline{U}$.
In particular, by
\autoref{proposition:smooth-morphism-reduction}(3),
we may now assume there is a generically smooth dominant map
$\overline{U} \ra \bp^{d-1}_k$. 

With notation as in \autoref{proposition:base-change-generic-fiber},
any tame Galois finite \'etale cover $E_L \ra U_L$ restricts to a cover
$E_{\eta_L} \ra U_{\eta_L}$ which is the base change of a tame Galois
finite \'etale cover
$F_{\eta_k} \ra U_{\eta_k}$.

Define $F$ to be the normalization of $U$ in the function field of $F_{\eta_k}$.
We claim that $F_L \simeq E_L$ as covers of $U_L$.
This will complete the proof, as it implies $F \to U$ is tame finite \'etale
and connected, since the same is true of $F_L \to U_L$.

To see $F_L \simeq E_L$ as covers of $U_L$, we know $E_L$ is the normalization of $U_L$ in $K(E_L) = K(E_{\eta_L})$.
Further, since $L/k$ has a separating transcendence basis (since $k$ is algebraically closed, hence perfect), 
it follows that $F_L$ is normal and has function field $K(E_L)$.
Moreover, the universal property of normalization induces a birational map $F_L \to E$.
Since both $F_L$ and $E$ are finite over $U_L$, the map $F_L \to E$ is finite.
It then follows from a version of
Zariski's main theorem that $F_L \to E$ is an
isomorphism \cite[\href{https://stacks.math.columbia.edu/tag/0AB1}{Tag
0AB1}]{stacks-project}.
\end{proof}

\subsection{Proof of injectivity in the smooth case}
\label{subsection:smooth-case}

Having verified the smooth quasi-projective case, we next verify the
smooth finite type and separated case.
The general idea is to use Chow's lemma to reduce to the projective case, but
there are a number of technical details.
We start by explaining the geometric consequence that Chow's lemma gives us.

\begin{lemma}
	\label{lemma:projective-codim-2-compactification}
	Suppose that $U$ is a smooth separated scheme of finite type over an
	algebraically closed field $k$ with a normal compactification
	$\alpha: U \to \overline{U}$.
	There is a closed subscheme $Z \subset U$ of codimension at least $2$ 
	and a normal projective compactification
	$\beta: U - Z \to X$ with a projective map
	$f: X \to \overline{U}$ so that $\alpha|_{U-Z} = f \circ \beta$.
\end{lemma}
\begin{proof}
	Using Chow's lemma, we can find a projective scheme $X$ with a
	birational projective map $f: X \to \overline{U}$, see
\cite[\href{https://stacks.math.columbia.edu/tag/0200}{Tag 0200}]{stacks-project}
and
\cite[\href{https://stacks.math.columbia.edu/tag/0201}{Tag
0201}]{stacks-project}.

We next construct a subscheme $Z \subset U$ of codimension at least $2$ and a
birational map $\beta: U - Z \to X$.
Since $f$ is birational, there is a dense open $W \subset U$ over
which $f$ is an isomorphism, so we obtain a map $g: W \to X$ which is an
isomorphism onto its image. Because $U$ is regular in codimension $1$ and $X$ is proper,
there is a scheme $Z \subset U$ of codimension at least $2$ so that $g: W \to
X$ extends to a birational map $\beta: U - Z \to X$.
Now, restricting $f$, we get a map $f' : f^{-1}(\alpha(U-Z)) \to U - Z$.

We claim $\beta$ factors through $f^{-1}(\alpha(U-Z))$ and thus defines a section to
$f'$.
Indeed, consider the composition $f \circ \beta: U - Z \to X \to \overline{U}$.
This agrees with $\alpha$ over the dense open $W$, and hence agrees with the
given open immersion $U - Z \to U \xrightarrow{\alpha} \overline{U}$ on $W$.
Because $U - Z$ is separated, $f \circ \beta$ must agree with the above open
immersion on all of $U - Z$. This implies that $\beta$ sends $U - Z$ to
$f^{-1}(\alpha(U-Z))$.

Let $\beta': U- Z \to f^{-1}(\alpha(U-Z))$ denote the map whose composition with
$f^{-1}(\alpha(U-Z)) \to X$ is $\beta$.
We will show next that $\beta'$ is a closed immersion.
We have seen above that $\beta'$ is a section to $f'$. Therefore, $\beta'$ is a monomorphism.
Moreover since $f'$ is projective, hence proper, $\beta'$ is also proper, as any
section to a proper map is proper via the cancellation theorem
\cite[10.1.19]{vakil:foundations-of-algebraic-geometry-2017} applied to the
composition $f' \circ \beta'$.
Since $\beta'$ is a proper monomorphism, it is a closed immersion
\cite[\href{https://stacks.math.columbia.edu/tag/04XV}{Tag
04XV}]{stacks-project}, hence projective.

We now conclude the proof.
By the above, the composition $U- Z \to f^{-1}(\alpha(U-Z)) \to X$ is
the composition of a closed immersion and an open immersion into a projective
scheme. This implies $U- Z$ is quasi-projective, and $U- Z \to X$ is a
normal projective compactification, as desired.
By construction, $\alpha|_{U-Z} = f \circ \beta$.
\end{proof}

We are now ready to reduce the proof of
\autoref{theorem:isomorphism-on-algebraically-closed-base-change} to the general
smooth case over an algebraically closed field,
which follows without much difficulty by applying the above lemma.
\begin{proposition}
	\label{proposition:smooth-finite-type-case}
	\autoref{theorem:isomorphism-on-algebraically-closed-base-change} holds
	when $U$ is smooth.
\end{proposition}
\begin{proof}
	Recall that $U$ is now smooth, finite type, and separated over $k =
	\overline k$ but not necessarily quasi-projective.
	Using Nagata compactification
\cite[\href{https://stacks.math.columbia.edu/tag/0F41}{Tag 0F41}]{stacks-project}
as described in \autoref{remark:tameness-definition},
we can find a normal compactification $\alpha: U \to \overline{U}$.
	By \autoref{lemma:projective-codim-2-compactification}, there is a
	closed subscheme $Z \subset U$ of codimension at least $2$ and a
	projective normal compactification $\beta: U - Z \to X$ with a
	projective map $f: X\to \overline{U}$ so that $\alpha|_{U - Z} = f \circ
	\beta$.

	For $Z \subset U$ of codimension at least $2$ as in
	\autoref{lemma:projective-codim-2-compactification},
	we have $\pt(U) \simeq \pt(U-Z)$ by
	\autoref{lemma:removing-codimension-2}.
	Therefore, it is enough to prove the theorem for the compactification $U
	-Z \to \overline U$. 
	By \autoref{lemma:dominating-compactification}, it is enough to prove
	the theorem for the compactification $U-Z \to X$ in place of $U - Z \to
	\overline U$.
	Finally, the theorem holds for the projective compactification $U - Z \to X$
	by \autoref{proposition:smooth-quasi-projective-case}.
\end{proof}

\subsection{Proof of injectivity in the general case}	
\label{subsubsection:normal-case}
We now complete the proof of the theorem for normal connected quasi-projective schemes,
using that we have proven it for
smooth $U$.
\begin{proof}[Proof of \autoref{theorem:isomorphism-on-algebraically-closed-base-change}]
	By \autoref{lemma:surjectivity}, the map $\pt(U_L) \ra \pt(U)$ is surjective.
	To complete the proof, we wish to show it is injective.
To verify the map $\pt(U_L) \ra \pt(U)$ is injective, by
\cite[Corollary 5.5.8]{szamuely:galois-groups-and-fundamental-groups},
it is enough to show that if $E \ra U_L$ is any connected finite \'etale cover,
then $E$ is isomorphic to $\widetilde{F}_L$ for $\widetilde{F} \ra U$ some connected
finite \'etale cover.
To see this, start with some $E \ra U_L$.
Let $W \subset U$ denote the maximal dense smooth open subscheme of $U$. 
Since we have already shown the map $\pt(W_L) \ra \pt(W)$ is an isomorphism
in \autoref{proposition:smooth-finite-type-case},
we know that $E \times_{U_L} W_L$ is isomorphic to the base change of some finite \'etale
cover
$F\ra W$ along $\spec L \ra \spec k$.
Let $\widetilde{F}$ denote the normalization of $U$ in $F$.
Since $U$ is normal, $\widetilde{F} \ra U$ is a finite morphism.
The setup this far is summarized by the commutative diagrams
\begin{equation}
	\nonumber
	\begin{tikzcd} 
		E \times_{U_L} W_L \ar {r} \ar {d} & E \ar {d} & F \ar {r} \ar {d} & \widetilde{F}  \ar {d} \\
		W_L \ar {r} & U_L  & W \ar {r} & U. 
	\end{tikzcd}\end{equation}

To complete the proof, we only need to show $\widetilde{F} \ra U$
is tame finite \'etale and there is an isomorphism $\widetilde{F}_L \simeq E$ over $U_L$.
Indeed, since $\widetilde{F}$ is normal and finite over $U$, the base change
$\widetilde{F}_L$ is also normal and finite over $U_L$.
It follows that $\widetilde{F}_L$ is the normalization of
$U_L$ in $F_L \simeq E \times_{U_L} W_L$.
But, since $E$ is also the normalization of $U_L$ in 
$E \times_{U_L} W_L$, we obtain that
$E \simeq \widetilde{F}_L$. 
Since $\widetilde{F}_L \simeq E \ra U_L$ is tame finite \'etale, it follows that
$\widetilde{F} \ra U$ is also tame finite \'etale, completing the proof. 
\end{proof}

\appendix
\section{Collected Lemmas}
\label{section:appendix}

In this appendix, we collect several lemmas used in the course of the above proof.
These are all quite standard, and we only include them for completeness. We include them in this appendix and
not in the body so as not to distract from the flow of the proof.

We begin with two standard results on how the tame fundamental group behaves upon passing to open subschemes.
These follow from the usual well-known versions for the full \'etale fundamental
group, but we spell out the usual proof for the reader's convenience.
\begin{lemma}
	\label{lemma:normal-open-surjective}
	Let $Y$ be a normal quasi-projective connected scheme and $W \subset Y$ be a nonempty open. Then the natural map
	$\pi_1(W) \to \pi_1(Y)$ is surjective. In particular,
	$\pt(W) \ra \pt(Y)$ is surjective, where tameness for $Y$ is taken with
	respect to a projective normal compactification $Y \to \overline{Y}$ and
	tameness for $W$ is taken with respect to $W \to Y \to \overline{Y}$.
\end{lemma}
\begin{proof}
	Assuming surjectivity of $\pi_1(W) \to \pi_1(Y)$,
	surjectivity of $\pt(W) \ra \pt(Y)$ follows from commutativity of the
	square
	\begin{equation}
		\label{equation:}
		\begin{tikzcd} 
			\pi_1(W) \ar {r} \ar {d} & \pi_1(Y) \ar {d} \\
			\pt(W) \ar {r} & \pt(Y)
	\end{tikzcd}\end{equation}
	and the fact that the vertical maps are surjective.

	It remains to verify $\pi_1(W) \to \pi_1(Y)$ is surjective. We need to check any connected finite \'etale cover $E \ra Y$
	has pullback $E \times_Y W$ which is also connected. 
	First, we claim $E$ is normal. Indeed, since normality is equivalent to being R1 and S2,
	$E$ is normal because the properties of being R1 and S2
	are preserved under \'etale morphisms.
	Therefore, $E$ is normal and connected, hence integral.
	Then, $E \times_Y W$ is a nonempty open subscheme of the integral scheme $E$, hence connected.
\end{proof}
For a proof of the next lemma in the case of fundamental groups, instead of tame
fundamental groups, see
\cite[Corollary 5.2.14]{szamuely:galois-groups-and-fundamental-groups}.
\begin{lemma}
	\label{lemma:removing-codimension-2}
	Let $U$ be a connected smooth $k$-scheme and $V \subset U$ a closed subscheme of codimension at least $2$.
	Then the natural map $\pt(U - V) \rightarrow \pt(U)$ is an isomorphism,
	where tameness for $U$ is taken with respect to a projective normal
	compactification $U \to \overline U$, and tameness for $U - V$ is taken
	with respect to $U- V \to U \to \overline{U}$.
\end{lemma}
\begin{proof}

The map is surjective by \autoref{lemma:normal-open-surjective}, so it suffices to verify injectivity.
	For this, we have to show that any tame finite \'etale cover 
 $E \ra U - V$ extends uniquely to a tame finite \'etale cover $E'$ of $U$.
 If $E \to U - V$ is tame, it follows from the definition of tameness and our
 compatible choices of compactifications that any extension will automatically
 also be tame.
 Hence, it suffices to show there is a unique extension.
Uniqueness is immediate because $E'$ is necessarily normal, and hence must be 
the normalization of $U$ in $E$.
So it suffices to check that the normalization $E'$ of $U$ in $E$ is a finite \'etale cover of $U$, restricting
to $E$ over $U - V$. 
That $E'$ restricts to $E$ over $U-V$ is clear and $E' \ra U$ is finite by finiteness of normalization.
Finally, $E'\ra U$ is \'etale by Zariski--Nagata purity
as in \cite[Exp.\,X, Th\'eor\`eme 3.1]{noopsortSGA1Grothendieck1971}
because it is \'etale over all codimension $1$ points and $U$ is smooth.
\end{proof}

Finally, we record a field-theory result on linear disjointness of certain extensions.
\begin{lemma}
	\label{lemma:linearly-disjoint}
	Suppose $k \ra L$ are algebraically closed fields.
	Let $k(x_1, \ldots, x_n) \to F$ 
	by any finite separable extension.
	Then $k(x_1, \ldots, x_n) \to F$ and $k(x_1, \ldots, x_n) \ra L(x_1, \ldots, x_n)$
	are linearly disjoint extensions.
\end{lemma}
\begin{proof}
We want to show the only finite separable extension of $k(x_1,\ldots, x_n)$ in $L(x_1,\ldots, x_n)$ is $k(x_1,\ldots, x_n)$.
To this end, let $F$ be some finite separable extension of $k(x_1,\ldots, x_n)$ in $L(x_1,\ldots, x_n)$.
So, to see $F$ is equal to $k(x_1,\ldots, x_n)$, it suffices to show
$F \otimes_{k(x_1,\ldots, x_n)} F$ is a domain.
We have a containment 
\begin{align*}
F \otimes_{k(x_1,\ldots, x_n)} F \subset L(x_1,\ldots, x_n) \otimes_{k(x_1,\ldots, x_n)} L(x_1,\ldots, x_n),
\end{align*}
so it suffices to show 
\begin{align*}
L(x_1,\ldots, x_n) \otimes_{k(x_1,\ldots, x_n)} L(x_1,\ldots, x_n)
\end{align*}
is a domain.
Indeed, this is a localization of 
\begin{align*}
L[x_1,\ldots, x_n] \otimes_{k[x_1,\ldots, x_n]} L[x_1,\ldots, x_n] \simeq (L\otimes_k L)[x_1,\ldots, x_n],
\end{align*}
so it suffices to show $L \otimes_k L$ is a domain.
This then holds because $L$ is a domain, and a domain over an algebraically
closed field is still a domain upon base change to any larger algebraically closed field,
i.e., the property of being geometrically integral is preserved under base change between algebraically closed fields.
\end{proof}

\bibliographystyle{alpha}
\bibliography{/home/aaron/Dropbox/master}

\def\cprime{$'$} \providecommand{\noopsort}[1]{}
\begin{thebibliography}{{\noopsort{SGA1}}R71}

\bibitem[BLR90]{BoschLR:Neron}
Siegfried Bosch, Werner L{\"u}tkebohmert, and Michel Raynaud.
\newblock {\em N\'eron models}, volume~21 of {\em Ergebnisse der Mathematik und
  ihrer Grenzgebiete (3) [Results in Mathematics and Related Areas (3)]}.
\newblock Springer-Verlag, Berlin, 1990.

\bibitem[FK88]{FreitagK:lectures-etale}
Eberhard Freitag and Reinhardt Kiehl.
\newblock {\em \'{E}tale cohomology and the {W}eil conjecture}, volume~13 of
  {\em Ergebnisse der Mathematik und ihrer Grenzgebiete (3) [Results in
  Mathematics and Related Areas (3)]}.
\newblock Springer-Verlag, Berlin, 1988.
\newblock Translated from the German by Betty S. Waterhouse and William C.
  Waterhouse, With an historical introduction by J. A. Dieudonn{\'e}.

\bibitem[GR71]{sga1}
A.~Grothendieck and M.~Raynaud.
\newblock {\em Rev\^etements \'etales et groupe fondamental}.
\newblock Springer-Verlag, Berlin-New York, 1971.
\newblock S{\'e}minaire de G{\'e}om{\'e}trie Alg{\'e}brique du Bois Marie
  1960--1961 (SGA 1).

\bibitem[Gro65]{EGAIV.2}
A.~Grothendieck.
\newblock \'{E}l\'ements de g\'eom\'etrie alg\'ebrique. {IV}. \'{E}tude locale
  des sch\'emas et des morphismes de sch\'emas. {II}.
\newblock {\em Inst. Hautes \'Etudes Sci. Publ. Math.}, (24):231, 1965.

\bibitem[Gro66]{EGAIV.3}
A.~Grothendieck.
\newblock \'{E}l\'ements de g\'eom\'etrie alg\'ebrique. {IV}. \'{E}tude locale
  des sch\'emas et des morphismes de sch\'emas. {III}.
\newblock {\em Inst. Hautes \'Etudes Sci. Publ. Math.}, (28):255, 1966.

\bibitem[Jou83]{jouanolou1982theoremes}
Jean-Pierre Jouanolou.
\newblock {\em Th\'{e}or\`emes de {B}ertini et applications}, volume~42 of {\em
  Progress in Mathematics}.
\newblock Birkh\"{a}user Boston, Inc., Boston, MA, 1983.

\bibitem[KS10]{kerzS:on-different-notions}
Moritz Kerz and Alexander Schmidt.
\newblock On different notions of tameness in arithmetic geometry.
\newblock {\em Math. Ann.}, 346(3):641--668, 2010.

\bibitem[Lan21]{landesman:geometric-average-selmer}
Aaron Landesman.
\newblock The geometric average size of {S}elmer groups over function fields.
\newblock {\em Algebra Number Theory}, 15(3):673--709, 2021.

\bibitem[Lit21]{litt2021arithmetic}
Daniel Litt.
\newblock Arithmetic representations of fundamental groups, {II}: {F}initeness.
\newblock {\em Duke Math. J.}, 170(8):1851--1897, 2021.

\bibitem[LO10]{lieblichO:generators-and-relations-for-the-etale-fundamental-group}
Max Lieblich and Martin Olsson.
\newblock Generators and relations for the \'{e}tale fundamental group.
\newblock {\em Pure Appl. Math. Q.}, 6(1, Special Issue: In honor of John Tate.
  Part 2):209--243, 2010.

\bibitem[LS57]{langS:sur-les-revetements}
Serge Lang and Jean-Pierre Serre.
\newblock Sur les rev\^{e}tements non ramifi\'{e}s des vari\'{e}t\'{e}s
  alg\'{e}briques.
\newblock {\em Amer. J. Math.}, 79:319--330, 1957.

\bibitem[Org03]{orgogozo:alterations-et-groupe-fondamental}
Fabrice Orgogozo.
\newblock Alt\'{e}rations et groupe fondamental premier \`a {$p$}.
\newblock {\em Bull. Soc. Math. France}, 131(1):123--147, 2003.

\bibitem[OV00]{orgogozo2000theoreme}
F.~Orgogozo and I.~Vidal.
\newblock Le th\'eor\`eme de sp\'ecialisation du groupe fondamental.
\newblock In {\em Courbes semi-stables et groupe fondamental en g\'eom\'etrie
  alg\'ebrique ({L}uminy, 1998)}, volume 187 of {\em Progr. Math.}, pages
  169--184. Birkh\"auser, Basel, 2000.

\bibitem[Sch02]{schmidt:tame-coverings}
Alexander Schmidt.
\newblock Tame coverings of arithmetic schemes.
\newblock {\em Math. Ann.}, 322(1):1--18, 2002.

\bibitem[{\noopsort{SGA1}}R71]{noopsortSGA1Grothendieck1971}
A.~{\noopsort{SGA1}}Grothendieck and M.~Raynaud.
\newblock {\em Rev\^etements \'etales et groupe fondamental}.
\newblock Springer-Verlag, Berlin-New York, 1971.
\newblock S{\'e}minaire de G{\'e}om{\'e}trie Alg{\'e}brique du Bois Marie
  1960--1961 (SGA 1).

\bibitem[{Sta}a]{stacks-project}
The {Stacks Project Authors}.
\newblock {\itshape Stacks Project}.
\newblock \url{http://stacks.math.columbia.edu}.

\bibitem[Stab]{MO:etale-fundamental-group-base-change-between-algebraically-closed-fields}
Jason Starr.
\newblock Is the map on etale fundamental groups of a quasi-projective variety,
  upon base change between algebraically closed fields, an isomorphism?
\newblock MathOverflow.
\newblock URL:http://mathoverflow.net/q/257762 (version: 2016-12-21).

\bibitem[Sza09]{szamuely:galois-groups-and-fundamental-groups}
Tam\'{a}s Szamuely.
\newblock {\em Galois groups and fundamental groups}, volume 117 of {\em
  Cambridge Studies in Advanced Mathematics}.
\newblock Cambridge University Press, Cambridge, 2009.

\bibitem[Vak]{vakil:foundations-of-algebraic-geometry-2017}
Ravi Vakil.
\newblock {\itshape MATH 216: Foundations of Algebraic Geometry}.
\newblock Available at \url{http://math.stanford.edu/~vakil/216blog/} November
  18, 2017 version.

\bibitem[{Zyw}10]{zywina2010hilbert}
D.~{Zywina}.
\newblock {Hilbert's irreducibility theorem and the larger sieve}.
\newblock {\em arXiv:1011.6465v1}, November 2010.

\end{thebibliography}

\end{document}